\documentclass[a4paper]{amsart}
\usepackage{amsfonts,amssymb,amscd,amsmath,enumerate,verbatim,newlfont,calc,}
\usepackage{amsmath}
\usepackage{amsmath, pb-diagram}
\usepackage{amsthm}
\usepackage{amsfonts}
\usepackage{amssymb}
\usepackage{latexsym}
\usepackage[all]{xy}

\newtheorem{thm}{Theorem}
\newtheorem{lem}[thm]{Lemma}
\newtheorem{prop}[thm]{Proposition}
\newtheorem{cor}[thm]{Corollary}

\newtheorem{note}[thm]{Remark}
\newtheorem{ex}[thm]{Example}

\def\z*{\mbox{Z$^*$}}
\def\ZM*{\mbox{Z$_M^*$}}
\def\-z{\mbox{$\overline{Z}_M$}}
\def\s*{\mbox{(S$^*$)}}

\begin{document}

\title[On two generalizations of perspective abelian groups ]{On two generalizations of \\ perspective abelian groups}

\author{Andrey R. Chekhlov}
\address{Faculty of Mathematics and Mechanics, Tomsk State University, Tomsk 634050, Russia}
\email{cheklov@math.tsu.ru}
\author{Peter V. Danchev}
\address{Institute of Mathematics and Informatics, Bulgarian Academy of Sciences, 1113 Sofia, Bulgaria}
\email{danchev@math.bas.bg}
\author{\"Ozg\"ur Ta\c{s}demir}
\address{Trakya University, Faculty of Economics and Administrative Science, Department of Business Administration, Balkan Campus, Edirne, T\"urkiye}
\email{ozgurtasdemir@trakya.edu.tr}

\date{\today}
\maketitle

\begin{abstract} We are generalizing in two non-trivial ways the recently defined {\it perspective} Abelian groups to the so-called {\it IC-groups} and {\it TP-groups}, respectively, and obtain numerous results in these two directions that can be viewed as improvements on their rather more complicated structures and properties.
\end{abstract}

\section{Introduction and Motivation}

Before starting our main work, we need some backgrounds as follows: For an arbitrary ring $R$, Nicholson and Sanchez call in \cite{morphic} an $R$-module $M$ to be \textit{morphic} if $M/\alpha(M)\cong Ker(\alpha)$ for all $\alpha\in End_R(M)$. They proved in \cite[Theorem 5]{morphic} that $M$ is a morphic module if, and only if, whenever $M/A\cong B$, where $A$ and $B$ are submodules of $M$, then $M/B\cong A$.\\

It is well known that an $R$-module $M$ is said to have \textit{the internal cancelation property} (or just {\it IC} for short) if whenever $M=K\oplus L=K'\oplus L'$ with $K\cong K'$, then $L\cong L'$.\\

Moreover, Ta\c{s}demir and Ko\c{s}an called in \cite{summand-morphic} an $R$-module $M$ to be \textit{summand-morphic} if $M/K\cong L$ where $K$, $L\leq^{\oplus}M$, then $M/L\cong K$. And they proved in \cite[Proposition 3.2]{summand-morphic} that such a module $M$ is summand-morphic if, and only if, $M$ has the IC property.\\

On the other hand, two direct summands $K$ and $L$ of an $R$-module $M$ are called \textit{perspective}, which is denoted by $K\overset{P}{\sim} L$, if there exists a submodule $N$ of $M$ such that $M=K\oplus N=L\oplus N$. Obviously, $K\overset{P}{\sim} L$ implies $K\cong L$ as well as $K\overset{P}{\sim} L$ $\iff$ $L\overset{P}{\sim} K$, which means that this operation is symmetric.\\

Besides, Garg, Grover and Khurana call in \cite{p} an $R$-module $M$ to be \textit{perspective} if $K\cong L$ yields that $K\overset{P}{\sim} L$ for any $K$, $L\leq^{\oplus} M$, i.e., there exists a submodule $N$ of $M$ such that $M=K\oplus N=L\oplus N$.\\

Strengthening this notion, Khurana and Nielsen call in \cite{tp} an $R$-module $M$ to have the \textit{transitive perspectivity} (or just {\it TP} for short) whenever $K\overset{P}{\sim} L\overset{P}{\sim} N$ forces $K\overset{P}{\sim} N$ for any $K$, $L$, $N\leq^{\oplus} M$.\\

It is easy to see that any perspective module is both IC and TP.\\

In the next statement, we provide an useful equivalent condition for being endowed with the TP property like this.

\begin{thm}\label{kurtarici}
The following are equivalent for an $R$-module $M$:
\begin{enumerate}
   \item $M$ has TP;
   \item If $M=K\oplus L=K'\oplus L'$ with $K\overset{P}{\sim}K'$, then $L\overset{P}{\sim}L'$.
\end{enumerate}
\end{thm}

\begin{proof}
$(1)\Rightarrow (2)$. Let $M$ have transitive perspectivity and $M=K\oplus L=K'\oplus L'$ with $K\overset{P}{\sim}K'$ for some submodules $K$, $K'$, $L$ and $L'$ of $M$. Then, one sees that $M=K\oplus N=K'\oplus N$ for some $N\leq M$. Now, we have $L\overset{P}{\sim} N$ and $N\overset{P}{\sim} L'$. Thus, $L\overset{P}{\sim} L'$, because $M$ has the TP property. \\
$(2)\Rightarrow (1)$. Letting $K$, $L$ and $N$ be direct summands of $M$ with $K\overset{P}{\sim} L$ and $L\overset{P}{\sim} N$, one can write that
$$M=K\oplus X=L\oplus X=L\oplus Y=N\oplus Y$$
for some $X$, $Y\leq M$. Therefore, $X\overset{P}{\sim} Y$ and, by assumption, it must be $K\overset{P}{\sim} N$, as required.
\end{proof}

Recently, C\v{a}lug\v{a}reanu and Chekhlov defined and explored in \cite{morphicab} the notion of {\it perspective abelian groups}. \\

Another advantage of this project is that C\v{a}lug\v{a}reanu published a paper about \textquotedblleft morphic abelian groups" before. We also have a similar equivalence to the definition of morphic for IC (please see the definition of summand-morphic which is equivalent to IC). This can make our work easier. It can allow us to make easier comparisons between morphic abelian groups and IC abelian gorups.\\

Our aim in the present article is to expand somewhat the concept of perspectivity for abelian groups to the next two more general concepts, again related to abelian groups, namely:

\medskip

\noindent\textbf{Internal Cancelation (IC):} We shall say that an abelian group $G$ is {\it IC} if $G=A\oplus B=A'\oplus B'$ with $A\cong A'$ assures that $B\cong B'$.\\

\noindent\textbf{Transitive Perspectivity (TP):} We shall say that an abelian group $G$ is {\it TP} if $G=A\oplus B=A'\oplus B'$ with $A\overset{P}{\sim}A'$ ensures that $B\overset{P}{\sim}B'$.\\

At first look on these two definitions, one can ask: are IC and TP properties equivalent for arbitrary modules? The answer is definitely {\bf no} as the following example shows. Recall that a module $M$ said to have the {\it substitution property} if, for every module $N$ having a decomposition $N=A\oplus B=A'\oplus B'$ with $A\cong M\cong A'$, there exists a submodule $C\leq N$ such that $N=C\oplus B=C\oplus B'$. Note that the substitution property yields the IC one.

\begin{ex}\label{noimplications}
(1) The $\mathbb{Z}$-module $\mathbb{Z}\oplus \mathbb{Z}$ has IC (see \cite[Example 3.2]{cigdem}), but it does not have the property of perspectivity by \cite[Theorem 5.12]{p}. Even something more, it is not TP in regard to \cite[Corollary 2.6]{tp} as obviously $\mathbb{Z}$ does not have the substitution property. \\
(2) In \cite[p.13]{handelman}, there is a construction which has the TP property, but is not directly finite (and hence, it does not have IC).
\end{ex}

However, for modules, the following hierarchy is principally well-known:\\

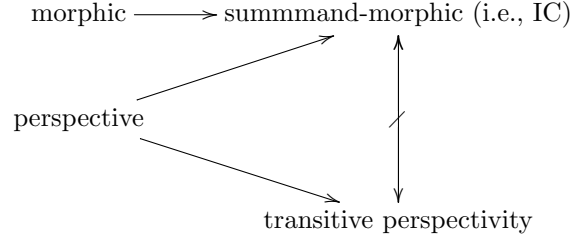
\noindent\begin{displaymath}
\xymatrix{
\text{morphic}    \ar[r]           &    \text{summmand-morphic (i.e., IC)}                         \\
\text{perspective}\ar[dr] \ar[ur]  &                                                                      \\
                                   &   \text{transitive perspectivity}\ar[uu] \ar@{<->}[uu]|-{\diagup} \\
         }
\end{displaymath}

\vskip 10pt

We now continue our work with two subsequent sections which contain the basic material motivating our writing of the present paper. Saying more conceptually, our objective here is to examine in-depth the IC and TP properties by discovering their fundamental properties.

\section{Abelian groups satisfying the internal cancelation property}

First, recall that the direct sum of two modules having the IC property need {\it not} to have the IC property (see, e.g., \cite{lam}). However, note that any indecomposable $R$-module has the IC property.

\medskip

For completeness of the exposition, we state the following version for abelian groups.

\begin{ex}
\begin{enumerate}
  \item The $\mathbb{Z}$-module $\mathbb{Q}$ is morphic, and hence possesses the IC property (see \cite{morphicab}).
  \item The $\mathbb{Z}$-module $\mathbb{Z}_n$ is morphic, and hence possesses the IC property (see \cite{morphic}).
  \item The $\mathbb{Z}$-module $\mathbb{Z}_{2}\oplus \mathbb{Z}_{3}$ is morphic, and hence possesses the IC property (see \cite{morphic}).
  \item The $\mathbb{Z}$-module $\mathbb{Z}_{p^\infty}$ possesses the IC property, but is not morphic (see \cite{morphic}).
  \item The $\mathbb{Z}$-module $\mathbb{Z}$ possesses the IC property, but is not morphic (see \cite{morphic}).
  \item The $\mathbb{Z}$-module $\mathbb{Z}\oplus \mathbb{Z}$ possesses the IC property by \cite[Example 3.2]{cigdem}, but is not morphic.
  \item The $\mathbb{Z}$-module $\mathbb{Z}_{2}\oplus \mathbb{Z}_{4}$ possesses the IC property, but is not morphic by \cite[p. 2645]{morphic}.
\end{enumerate}
\end{ex}

The next statement is worthy of recording.

\begin{prop}\label{Q}
Any finite direct sum of $\mathbb{Q}$ (equivalently, a finite rank torsion-free divisible group) is IC.
\end{prop}

\begin{proof}
It is immediate invoking \cite[Proposition 3]{morphicab}.
\end{proof}

But, as opposite to Proposition~\ref{Q}, an infinite direct sum of copies of $\mathbb{Q}$ as a $\mathbb{Z}$-module need {\it not} be an IC group (see \cite[Example 20]{unit-regular}).\\

It was proven in \cite[Theorem 4]{morphicab} that, if a group $G$ is morphic, then $G$ is divisible exactly when it is torsion-free.

\medskip

However, the next example illustrates that this assertion is {\it no} longer true for IC abelian groups.

\begin{ex}
\begin{enumerate}
  \item The $\mathbb{Z}$-module $\mathbb{Z}$ is a torsion-free IC group, but it is not divisible.
  \item The $\mathbb{Z}$-module $\mathbb{Z}_{p^\infty}$ is a divisible IC group, but it is not torsion-free.
\end{enumerate}
\end{ex}

The following claim is helpful.

\begin{prop}
Any free abelian group with finite rank is IC.
\end{prop}

\begin{proof}
Suppose $A$ and $B$ are finite direct summands of $G$ with $A\cong B$. For some subgroups $A'$ and $B'$ of $G$, we may write $G=A\oplus A'=B\oplus B'$. Then, $$r(G)=r(A)+r(A')=r(B)+r(B').$$ Since $A\cong B$, we deduce $r(A)=r(B)$ and, therefore, $r(A')=r(B')$. Thus, in virtue of \cite[Theorem 10.14]{rotman}, we infer the desired isomorphism $A'\cong B'$.
\end{proof}

In order to formulate our other achievements, we need the following.

\begin{lem}\label{hom}
If $R$-modules $M$ and $N$ both have the IC property such that $$Hom_R(M,N)=\{0\}=Hom_R(N,M),$$ then $M\oplus N$ has the IC property too.
\end{lem}

\begin{proof}
It is follows at once from \cite[Proposition 8.5]{lam}.
\end{proof}

We, thereby, arrive at the following (see \cite[Proposition 5]{morphicab} too).

\begin{prop}\label{torsion}
A torsion abelian group is IC if, and only if, all its primary components are IC.
\end{prop}

\begin{proof}
The necessity is pretty clear, because being IC is inherited via direct summands.\\
For the sufficiency, let $G=\bigoplus_p G_p$ be the direct decomposition of $G$ into $p$-components. Likewise, each subgroup $N$ of $G$ has a corresponding decomposition into primary components $N=\bigoplus_p N_p$, with $N_p$ a subgroup of $G_p$ for all $p$. So, for a direct summand $K$ of $G$, we can write $G/K\cong \bigoplus_p (G_p/K_p)$.
Now, we assume that $G/K\cong L$ for a direct summand $L$ of $G$. Similarly, we may write $L=\bigoplus_p L_p$, whence $\bigoplus_p (G_p/K_p)\cong \bigoplus_p L_p$. Then, it follows that $G_p/K_p\cong L_p$ as the property is inherited component-wisely by direct summands. But, since every $G_p$ is IC, it must be that $G_p/L_p\cong K_p$, and hence it is immediately true that $\bigoplus_p (G_p/L_p)\cong \bigoplus_p K_p$, as wanted.
\end{proof}

We proceed by proving the following.

\begin{lem}\label{injec}
Any injective module with finite Goldie dimension has the IC property.
\end{lem}

\begin{proof}
It is a direct consequence of \cite[Proposition 17]{wp}.
\end{proof}

Before giving the next result, let us note that any module having the IC property is necessarily directly finite.

\begin{prop}\label{div}
Let $G$ be a divisible abelian group. Then, $G$ is IC if, and only if, all $p$-primary components and their torsion free parts have finite uniform dimension (i.e., finite Goldie dimension).
\end{prop}

\begin{proof}
$(\Rightarrow)$: Let $G$ be IC. Then, every $p$-primary components of $G$ is also IC as well as their torsion-free parts are IC, because they are direct summands; thus, by what we have note above, they are directly finite. Therefore, every $p$-primary component and its corresponding torsion free part will have finite uniform dimension (i.e., finite Goldie dimension).\\

$(\Leftarrow)$: We prove the sufficiency for two different cases:

Firstly, we assume that $G$ is a torsion group. Then, the proof is same as that of Proposition \ref{torsion}.

Now, we assume that $G$ is not torsion. Then, we may write that $G=G_T\oplus G_F$, where $G_F=\bigoplus^{n}_{i=1}\mathbb{Q}$ and $G_T$ is the torsion part of $G$. Appealing to Lemma \ref{injec}, both $G_T$ and $G_F$ are IC groups. But, since $Hom(G_T,G_F)=\{0\}$ and $G_T$,$G_F$ both are IC, Lemma \ref{hom} applies to get that $G$ is too IC, as promised.
\end{proof}

\begin{note}
Recall that a direct sum of cyclic $p$-groups is called \textit{homogeneous} if all its direct summands of rank 1 are isomorphic. It is evident then that homogeneous groups are themselves IC whenever they have finite rank (compare with Remark~\ref{0.1} alluded to below). On the other hand, it was proven in \cite[Theorem 7]{morphicab} that a (reduced) $p$-group $G$ is morphic precisely when it is finite and homogeneous.
\end{note}

The proofs of the next result is analogous to \cite[Proposition 10]{morphicab}.

\begin{prop}
Suppose that $G=D(G)\oplus R$, where $D(G)$ is the divisible part of a group $G$. If both $D(G)$ and $R$ are IC, and $R$ is torsion, then $G$ is a splitting IC group.
\end{prop}

\begin{proof}
Under our assumption, we may write $G=D(G)\oplus T(G)$. Observe that $$Hom_\mathbb{Z}(divisible, reduced)=\{0\}=Hom_\mathbb{Z}(torsion, torsion\text{-}free).$$ Now, the claim follows automatically from
Lemma~\ref{hom}.
\end{proof}

One may emphasize that an absolute analogy to \cite[Proposition 11]{morphicab} is, however, {\it no} longer possible since as demonstrated above there is an unbounded, and thus an infinite, IC $p$-group.

\medskip

%\begin{prop}
%Suppose $G$ is a group with maximal torsion subgroup $T(G)$. If both $T(G)$ and $G/T(G)$ are IC, and $T(G)$ has only finitely many $p$-components, then $G$ is a splitting IC group.
%\end{prop}

%\begin{proof}
%Under the given restrictions on $T(G)$, it is finite ??? and hence $G$ splits utilizing the classical Baer-Fomin's celebrated theorem. However, $G/T(G)$ is IC and hence it is both torsion-free and divisible. So, the proof argues again as in Lemma~\ref{hom}.
%\end{proof}

The following extra comments are worthwhile.

\begin{note}\label{0.1} (1) (see, for a more account, \cite[Nonexample 3.1.]{perspective-ab}) If $A$ is a module such that $A\ncong A\oplus A$ and $\alpha$
is an infinite ordinal. Then, both $A^{(\alpha)}$ and $A^{\alpha}$ are {\it not IC}.

(2) A direct summand of an IC module is too IC.
\end{note}

\begin{proof}
(1) One simple inspects that the subgroups $H=A^{(\alpha-1)}$ and $S=A^{(\alpha-2)}$ are isomorphic direct summands of $A^{(\alpha)}$, but it can plainly verified that $$A^{(\alpha)}/H\cong A\ncong A\oplus A\cong A^{(\alpha)}/S,$$ as required. For $A^{\alpha}$ the arguments are quite similar, so we omit the details.

(2) Suppose $G$ is an IC module, and write $G=A\oplus B$ and $A=K\oplus L=N\oplus F$, where $K\cong N$. Then, it must be that $$G=(K\oplus B)\oplus L=(N\oplus B)\oplus F.$$ Since $K\oplus B\cong N\oplus B$, one sees that $L\cong F$, as needed.
\end{proof}

Now, recall that, if the $p$-group $G$ is a direct sum of cyclic groups, then one may write that $G=\bigoplus_{n\geq 1}G_n$, where $G_n\cong\bigoplus_n \mathbb{Z}(p^n)$ are the {\it homogeneous} components of $G$.

\medskip

We now intend to establish the following two criteria.

\begin{prop}\label{0.2} (i) The fully decomposable torsion-free group $G$ is IC if, and only if, all homogeneous components of $G$ have finite ranks.

(ii) The direct sum of cyclic $p$-groups $G$ is IC if, and only if, all homogeneous components of $G$ have finite ranks.
\end{prop}

\begin{proof}
(i) The necessity follows at once from Remark~\ref{0.1}.

As for the sufficiency, write $$G=K\oplus L=N\oplus F,$$ where $K\cong N$. As noted above, we have $G=\bigoplus_{i\in I}G_i$, where all $G_i$ are the homogeneous components of $G$. As the fully decomposable groups have the property of isomorphic extensions of direct decompositions (cf. \cite[\S 86, Exersize 8]{F}), it thus follows that $L=\bigoplus_{i\in I}L_i$ and $F=\bigoplus_{i\in I}F_i$, where both $L_i$ and $F_i$ are isomorphic to corresponding direct summands in $G_i$. Since $K\cong N$, $r(L_i)=r(F_i)$, so it must be that $L_i\cong F_i$ for all $i$, whence $L\cong F$, as pursued.

The proof of point (ii) is pretty analogous to that of point (i), because the direct sum of cyclic groups also has the property of isomorphic extensions of direct decompositions (see, e.g., \cite[Corollary 18.2]{F}).
\end{proof}

As a valuable consequence, we yield:

\begin{cor}
Any fully decomposable torsion-free group of finite rank is IC. Additionally, IC is also each finite group, and any IC bounded group is finite.
\end{cor}

In Proposition~\ref{0.2} (ii) we described IC groups which are direct sums of cyclic $p$-groups. That is why, it follows from Lemma~\ref{hom} and Proposition~\ref{torsion} that an arbitrary direct sum of cyclic groups is IC if, and only if, its torsion part has finite rank and each $p$-component is IC.

\medskip

We now proceed by showing the following interesting note.

\begin{note}\label{torfree}
Every torsion-free group of rank 2 is IC since it is either decomposable or indecomposable. But there exists a torsion-free group of rank 3 which is {\it not} IC. Indeed, an appeal to \cite[Theorem 90.4]{F} gives that, for every integer $m\geq 3$, there is a torsion-free group $G$ of rank 3 such that $G=B_i\oplus C_i$ with $i=1,\dots, m$, where all $B_1,\dots, B_m$ are isomorphic groups of rank 1, and all $C_1,\dots, C_m$ are pair-wise non-isomorphic indecomposable groups of rank 2.
\end{note}

Note that, for every IC $p$-group $G$, it follows that each Ulm invariant $f_n(G)$ is finite for $n<\omega$; in particular, $|G|\leq 2^{\aleph_0}$.

\medskip

We are now intending to establish a series of characteristic assertions.

\begin{prop}
Let $G$ be a reduced countable $p$-group having finite all Ulm-Kaplansky invariants. Then, $G$ is IC.
\end{prop}

\begin{proof}
Write $G=K\oplus L=N\oplus F$, where $K\cong N$. In view of the finiteness of $f_{\alpha}(L)=f_{\alpha}(F)$, holding for every ordinal number $\alpha$, and in virtue of the countability of $G$, one readily follows from \cite[Theorem 77.3]{F} that $L\cong F$, as required.
\end{proof}

\begin{prop}
The reduced torsion-complete $p$-group $G$ is IC if, and only if, every its $f_n(G)$ is finite for any $n<\omega$.
\end{prop}

\begin{proof}
Write $G=K\oplus L=N\oplus F$, where $K\cong N$. Since $f_{n}(L)=f_{n}(F)$, it follows that the basic subgroups of $L$ and $F$ are isomorphic, and hence isomorphic are also the whole groups $L$ and $F$ (see, e.g., \cite{F2}).
\end{proof}

Recall that, if $G$ is a torsion-free group, then its $p$-rank $r_p(G)$ is exactly the rank of the factor-group
$G/pG$.

\begin{prop}
The reduced algebraically compact torsion-free group $G$ is IC if, and only if, for each prime $p$, the $p$-rank
$r_p(G)$ is finite.
\end{prop}

\begin{proof}
{\bf Necessity.} Write $G=\prod_pG_p$, where each $G_p$ is the $p$-adic component of $G$ which is a $p$-adic module.
Notice that the $p$-rank of $G$ coincides with the $p$-rank of $G_p$, and thus it coincides with the rank of a basic submodule of $G_p$ as a module over the ring of integer $p$-numbers $\mathbb{\widehat{Z}}_p$.

If, however, its $p$-rank is infinite, then one writes that
$$G_p=\mathbb{\widehat{Z}}^+_p\oplus K=\mathbb{\widehat{Z}}^+_p\oplus\mathbb{\widehat{Z}}^+_p\oplus C,$$
where $K\cong C$, which fact is manifestly imposable for IC modules, thus substantiating our claim.

{\bf Sufficiency.} Write $G=K\oplus L=N\oplus F$, where $K\cong N$. As already observed above, the finiteness of the $p$-rank for each $G_p$ insures that the basic subgroups of $L_p$ and $F_p$ are isomorphic, whence isomorphic are $L$ and $F$ as well (cf. \cite{F2}).
\end{proof}

Recall that any vector torsion-free group $G$ can be written as $G=\prod_t G_t$, where all direct components $G_t$
are the direct product of groups of rank 1 and type $t$ (see, for instance, \cite{N}).

\begin{prop}
The reduced vector torsion-free group $G$ is IC if, and only if, all its direct components $G_t$ have finite rank.
\end{prop}

\begin{proof}
Since each $G_t$ has finite rank, it must be that $|G|\leq 2^{\aleph_0}$, and so $G$ is non-measurable, which enables us that every direct summand of $G$ also is a vector group (see, respectively, \cite[Theorem 5.6]{F2} or \cite{N}).
Therefore, if we write $$G=K\oplus L=N\oplus F,$$ where $K\cong N$, then $$K=\prod_tK_t, L=\prod_tL_t, N=\prod_tN_t,
F=\prod_tF_t.$$
However, thanks to \cite[Theorem 5.5]{F2}, we obtain that $$G_t\cong K_t\oplus L_t\cong N_t\oplus F_t.$$
Furthermore, with the condition on finite ranks at hand, we arrive at $L_t\cong F_t$ for all $t$, so that it is immediate $L\cong F$, as wanted.
\end{proof}

\begin{note}
Recall that the group $B$ has the cancelation property (that is, a {\it CP group}) if, for all groups $H,K$, the isomorphism $B\oplus H\cong B\oplus K$ forces the isomorphism $H\cong K$. This is obviously equivalent to the next property: for all groups $H,K$, from the equality $B\oplus H=C\oplus K$, the isomorphism $B\cong C$ yields the isomorphism $H\cong K$ (see the commentaries before \cite[Theorem 90.4]{F}). However, looking at \cite[\S15, Exersize 24]{F}, one finds that all finitely generated groups are CP. Moreover, all finitely co-generated groups also are CP.
This simply leads to the fact that the quasi-cyclic groups are CP, but we may directly derive that from Proposition~\ref{07} quoted below.
\end{note}

We now would like to indicate the following elementary, but pivotal properties.

\medskip

(1) A direct summand of a CP group $A$ is too CP.

Indeed, if $A=B\oplus C$ and $B\oplus H\cong B\oplus K$, then $$(B\oplus C)\oplus H\cong (B\oplus C)\oplus K,$$
so $H\cong K$, as required.

\medskip

(2) If $A$ is a CP group, then $A^n$ also is a CP group for every integer number $n\geq 2$.

We argue by induction: if $$(A\oplus A)\oplus H\cong(A\oplus A)\oplus K,$$ then $$A\oplus (A\oplus H)\cong A\oplus (A\oplus K),$$ and thus $A\oplus H\cong A\oplus K$ yielding $H\cong K$, as needed.

\medskip

Even something more, if $H_1$ and $H_2$ are both CP groups, then $H_1\oplus H_2$ also is CP, i.e., the class consisting of CP groups is closed under formation of finite direct sums. To substantiate this, it is pretty clear that it is enough to settle the property only for two direct summands: to that end, supposing $H_1\oplus H_2\oplus K\cong H_1\oplus H_2\oplus M$, then $H_2\oplus K\cong H_2\oplus M$. Thus, $K\cong M$, that is, $H_1\oplus H_2$ is CP, as required.

\medskip

(3) Same as Remark~\ref{0.1}, it can be proven that, if $A\ncong A\oplus A$, then both $A^{(\alpha)}$ and $A^{\alpha}$ are not CP, provided the cardinal $\alpha$ is infinite. Indeed, if $A$ is a non-zero CP group, then it must be that $A\oplus A\ncong A$ -- in fact, if $A\oplus A\cong A$, then $A\oplus A\cong A\oplus\{0\}$, so that $A\cong\{0\}$, a contradiction, as suspected.

\medskip

(4) Any CP group $A$ is IC.

Indeed, if we write $$A=B\oplus C=H\oplus F,$$ where $B\cong H$, then one follows that
$$(B\oplus C)\oplus C=(H\oplus F)\oplus C=(H\oplus C)\oplus F,$$
where $$B\oplus C\cong H\oplus C\cong A,$$ forcing $C\cong F$.

\begin{note}\label{notenew}
Not every rank 1 torsion-free group is CP (see, for example, Remark~\ref{torfree}), so Proposition~\ref{0.2} gives the desired example of a IC group that is {\it not} CP. In fact, each torsion-free group of finite rank is CP in the class of torsion-free groups of rank 1, as indicated in the second bullet below.
\end{note}

We now proceed with the following two more relevant observations:

\medskip

\noindent $\bullet$ ~ If $A$ and $B$ are isomorphic divisible subgroups of some group $G$ with $A\cap B=\{0\}$, then $A$ and $B$ have common addition in $G$. Indeed, we can write $G=(A\oplus B)\oplus C$. If $f: A\to B$ is the existing isomorphism, then one may write that $A\oplus B=A\oplus H=B\oplus H$, where $H=\{a+f(a)\,|\,a\in A\}$, whence it must be that $$G=B=A\oplus H\oplus C=B\oplus H\oplus C,$$ as suspected.

\medskip

\noindent $\bullet$ ~ Every torsion-free group $G$ of finite rank is CP in the class of torsion-free groups of rank 1. Indeed, if we suppose that $K\oplus G\cong H\oplus A$, where $G\cong A$ and $K$, $H$ are torsion-free of rank 1, then we know that $K\cong H$ (cf. \cite[\S 90, Exersize 13]{F}). But, as already observed above, not every torsion-free group of finite rank is TP.

\medskip

We are now preparing to describe all perspective torsion groups having divisible part with each primary component of finite rank and reduced part with each primary component of finite Ulm-Kaplansky invariants.

\begin{thm} A reduced countable $p$-group is perspective if, and only if, it has finite all Ulm-Kaplansky invariants.
\end{thm}

\begin{proof} Necessity is manifestly evident, so we leave out the details.

To substantiate sufficiency, let us write $G=A\oplus B=C\oplus K$, where $A\cong B$. Notice that any $p$-group with finite all Ulm-Kaplansky invariants always possesses the substation property (cf. \cite{F71}, \cite[Problem 58]{F2}, or \cite{C65}). Since $A$ also possesses finite all Ulm-Kaplansky invariants, it has the substation property, so that $G=C\oplus B=C\oplus K$ for some $C\leq G$. In particular, it must be that $B\cong K$, and hence
we freely write $G=A\oplus U=C\oplus U$ for some $U\leq G$, as expected.
\end{proof}

We terminate this section with the following statement.

\begin{prop}\label{07}
If $B\oplus G=C\oplus H$, where both $B,C$ are isomorphic to divisible IC groups, then $G\cong H$.
\end{prop}

\begin{proof}
Note that $B,C$ are also divisible IC groups, and so all their $p$-primary components and their torsion-free parts must have finite (possibly different) ranks. In fact, if some of these ranks is supposed, for a moment, to be infinite, then there will be a situation when $B\oplus G=C\oplus\{0\}$, where $B,C,G\neq\{0\}$ and $B\cong C$ - an absurd. This unambiguously means that, if $G$ is reduced, then $H$ also has to be reduced and, therefore, in this case $G\cong H$; likewise, if $G,H$ are non-reduced, then their divisible parts $D(G)$ and $D(H)$ are isomorphic, so that if we write $G=D(G)\oplus R(G)$ and $H=D(H)\oplus R(H)$, then $R(G)$ and $R(H)$ are isomorphic to the reduced part of $B\oplus G$. Thus, $R(G)\cong R(H)$ and, consequently, $G\cong H$, as asked for.
\end{proof}

\section{Abelian groups satisfying transitive perspectivity}

We begin our work here by recalling some crucial facts as follows: First, we recall that a group satisfies the {\it weak perspectivity} property, or just saying that it is {\it weakly perspective}, that is {\it WP} for short, provided all of its isomorphic direct summands with isomorphic complements are perspective (see \cite{wp}). Moreover, a group $M$ is called {\it purely infinite}, provided $M\cong M\oplus M$.

\medskip

The following helpful observations are fulfilled.

\begin{note} (1) A group is perspective if, and only if, it is simultaneously WP and IC (\cite[Remark 8]{wp});
we emphasize that $\mathbb{Z}\oplus\mathbb{Z}$ is IC but {\it not} WP.

(2) If $M$ is a non-zero purely infinite group, then $M$ is {\it not} WP (\cite[Proposition 9]{wp}), so $M$ is a group such that $M\ncong M\oplus M$, and if $\alpha$ is an infinite ordinal, then both $M^{(\alpha)}$ and $M^{\alpha}$ are {\it not} WP.

3) A direct summand of a WP group is too WP (see \cite[Lemma 6]{wp}).
\end{note}

Note also that, if $M$ is a quasi-injective and WP group, then $M$ is perspective (see \cite[Theorem 16]{wp}).

In addition, the following is valid.

\begin{note} (1) The fully decomposable torsion-free WP group $G$ is always perspective. In fact, knowing point (2) of the preceding Remark and Proposition~\ref{0.2} (1), we conclude that $G$ is IC, so it has to be perspective with the help of point (1) of the previous Remark.

(2) It can be proven by analogy that each of the bottom groups, which are assumed to be WP groups:

(a) the direct sums of cyclic $p$-groups;

(b) the reduced countable $p$-groups having finite all Ulm-Kaplansky invariants;

(c) the reduced torsion-complete $p$-groups;

(d) the reduced algebraically compact torsion-free groups;

(e) the reduced vector torsion-free groups,

are always perspective groups.
\end{note}

Recall also that a module $M$ is called a {\it CS module} if every submodule of $M$ is essential in a direct summand of $M$. Also, a module $M$ is called a {\it C3 module} if, for any module $K\leq^{\oplus} M$, the conditions $L\leq^{\oplus} M$ with $K\cap L=\{0\}$ imply $K\oplus L\leq^{\oplus} M$. And, a module is called {\it quasi-continuous} if it is both CS and C3.

\medskip

We are now managing to prove the following reduction criterion.

\begin{prop}\label{torsion}
Let $M$ be a quasi-continuous group. Then, $M$ is TP if, and only if, all $p$-primary components and their torsion free parts have finite uniform dimension (i.e., finite Goldie dimension).
\end{prop}

\begin{proof}
$(\Rightarrow)$: Suppose $M$ is TP. In conjunction with \cite[Theorem 3.5]{tp} and \cite[Proposition 2.22]{u-moduletp}, $M$ is weakly perspective. Thus, all $p$-primary components of $M$ are weakly perspective. Likewise, their torsion free parts are weakly perspective (note that their torsion free parts are actually direct summands) and, clearly, they are Dedekind finite. Therefore, all $p$-primary components and their torsion free parts have finite uniform dimension (i.e., finite Goldie dimension), as formulated.\\

$(\Leftarrow)$: Suppose that $M$ is a torsion module with $M=K\oplus K'=L\oplus L'$ and $K\cong L$, $K'\cong L'$.
So, $M=\bigoplus M_p$, where $M_p$ means the $p$-primary component of $M$ for some prime number $p$. Now, we have $K=\bigoplus K_p$, $L=\bigoplus L_p$, $K'=\bigoplus K'_p$ and $L'=\bigoplus L'_p$. Consequently, $$M_p=K_p\oplus K'_p=L_p\oplus L'_p$$ with $K_p\cong L_p$ and $K'_p\cong L'_p$. However, bearing in mind,
\cite[Proposition 2.27]{u-moduletp}, one has that $M_p=K_p\oplus N_p=L_p\oplus N_p$ for some $N_p\leq M_p$. Setting $N:=\bigoplus N_p$ it must be that $M=K\oplus N=L\oplus N$. If we assume for a moment that $M$ is not torsion, then $M=M_T\oplus M_F$, where $M_F=\bigoplus^{n}_{i=1}\mathbb{Q}$ and $M_T$ is the torsion part of $M$. But, since $Hom(M_T,M_F)=\{0\}$, and $M_T$ and $M_F$ are weakly perspective, taking into account \cite[Corollary 24]{wp}, $M$ has to be weakly perspective. Finally, looking at \cite[Theorem 3.5]{tp}, %and \cite[Proposition 2.22]{u-moduletp},
we may derive that $M$ is TP, as stated.
\end{proof}

With this at hand, one can infer the following consequence.

\begin{cor}\cite[Corollary 18]{wp}
Let $M$ be a divisible group. Then, $M$ is a weakly perspective module if, and only if, all $p$-primary components and their torsion free parts have finite uniform dimension (i.e., they have finite Goldie dimension).
\end{cor}

It was observed after Proposition~\ref{Q} that an infinite direct sum of copies of $\mathbb{Q}$ as a $\mathbb{Z}$-module need {\it not} be an IC group. Besides, this is also an example of a group that is {\it not} TP as well. Indeed, the following arguments are valid:

\medskip

\noindent\textbf{Fact 1:} (Folklore due to Bass) A ring $R$ is Noetherian if, and only if, every direct sum of injective $R$-modules is injective.

\medskip

\noindent\textbf{Fact 2:} In divisible groups, if a group is {\it not} IC, it can {\it not} be TP (see Lemma \ref{lemmaqc} listed below).

\medskip

\noindent\textbf{Fact 3:} In divisible groups, perspective groups and TP groups do coincide by \cite[Theorem 3.5]{tp}.

\medskip

We now arrive at the following two constructions.

\begin{ex}
Let $M$ be an infinite direct sum of copies of $\mathbb{Q}$ as a $\mathbb{Z}$-module. It was showed in \cite[Example 20]{unit-regular} that $M$ is not IC. Using Fact 1, $M$ is a divisible group. Utilizing Fact 2, $M$ can not be TP, as expected.
\end{ex}

\begin{ex}
Let $M$ be an infinite direct sum of copies of $\mathbb{Z}_{p^\infty}$ as a $\mathbb{Z}$-module. Then, $M$ is not IC. With the same arguments in hand, $M$ can not be TP, as suspected.
\end{ex}

On the other hand, let $M$ be the infinite direct sum copies of $\mathbb{Z}$ as a $\mathbb{Z}$-module. Thus, $M$ is not IC. But does it follow that it is TP or not?\\

We also remember the following.

\medskip

\textbf{Fact 4:} (\cite[Proposition 3.2(3)]{df}) If $M$ is the direct sum of infinitely many copies of the same non-zero module, then $M$ is {\it not} a DF module (and hence $M$ is {\it not} IC).\\

As another, alternative, example, what happens with the $\aleph_0$ direct sum of the integers $\mathbb{Z}$? What we surely know so far is that it is neither IC {\it nor} TP, because a direct summand of TP remains TP, as Lemma~\ref{tp01} unambiguously shows below, and $\mathbb{Z}\oplus \mathbb{Z}$ is known to be {\it not} TP.

\medskip

We are now prepared to establish some further critical properties of TP groups.

\begin{lem}\label{tp01}
A direct summand of a TP group is again a TP group.
\end{lem}

\begin{proof}
Suppose $A$ is a direct summand of a TP group $G$. One routinely sees that, if $K$ and $V$ are direct summands of
$A$ and $G=K\oplus W=V\oplus W$, then one may write $A=K\oplus(A\cap W)=V\oplus(A\cap W)$, as required.
\end{proof}

\medskip

Our established so far machinery allows us to obtain the following assertion.

\begin{cor}
Each homogeneous TP $p$-group is perspective (i.e., it has finite rank).
\end{cor}

\begin{proof}
The claim is true by combining Lemma~\ref{lemmaqc} situated below with the fact that $G$ is fully invariant and essential in the divisible hull $D(G)$, so the required properties of the direct summands of $G$ and $D(G)$ are the same.
\end{proof}

The big question, which seems to be extremely difficult, is to find if possible a reduced TP group that is {\it not} perspective (see Lemma~\ref{lemmaqc} below for the divisible case).

\medskip

Recall that a divisible group $D$ is perspective if, and only if, the torsion-free rank $r_0(D)$ is finite and each $p$-primary component of $D$ has finite rank $r_p(D)$; Proposition~\ref{div} teaches us that this is equivalent to the requirement that $D$ is IC.

\medskip

The following claim sheds a further light on the properties of divisible perspective groups.

\begin{prop}
Every divisible perspective group possesses the substitution property in the class of groups in which the divisible part is perspective.
\end{prop}

\begin{proof}
Suppose $G$ is such a group, and write $N=A\oplus B=A'\oplus B'$, where $A\cong G\cong A'$.

If, for a moment, $B$ is reduced, then since both $r_0(A)$,$r_p(A)$ are finite and $A\cong A'$, then $B'$ also is reduced, so $A=A'$.

Assume now that $B$ is not reduced, and write $B=D(B)\oplus B_1$, $B'=D(B')\oplus B'_1$, where $B_1$,$B_1'$ are respectively the reduced parts of $B$,$B'$. Thus, it must be that
$$N=(A\oplus D(B))\oplus B_1=(A'\oplus D(B'))\oplus B_1',$$ putting $D:=A\oplus D(B)=A'\oplus D(B')$. Since both pairs $r_0(D(B))$, $r_p(D(B))$ and $r_0(D(B'))$, $r_p(D(B'))$ are all finite as well as $r_0(A)=r_0(A')$, $r_p(A)=r_p(A')$
are too finite for all prime $p$, one extracts that $D(B)\cong D(B')$. By assumption, the group $D$ is perspective, leading to $$D=U\oplus D(B)=U\oplus D(B').$$ Finally, $$N=U\oplus(D(B)\oplus B_1)=U\oplus(D(B')\oplus B_1'),$$ as required.
\end{proof}

It is worthwhile noticing that, we established after Lemma~\ref{lemmaqc} situated below that, all the ranks $r_0(G)$ and $r_p(G)$ of the group $G$ from the last proposition can {\it not} be infinite.

\medskip

We now need the following technicality.

\begin{lem}\label{tp02} (\cite[Lemma 9.5]{F})
Let $A=B\oplus C$ be a direct composition with corresponding projections $\pi$, $\theta$. If for the direct decomposition $A=B\oplus C_1$ the corresponding projections are, respectively, $\pi_1$ and $\theta_1$, then
$$\pi_1=\pi+\pi\varphi\theta, \ \theta_1=\theta-\pi\varphi\theta$$
for some endomorphism $\varphi$ of the subgroup $A$.\\
Conversely, if the endomorphisms $\pi_1$ and $\theta_1$ have the above form, then $A=B\oplus\theta_1A$.
\end{lem}

For the reader's convenience, we remember that if $A\leq G$, then the {\it co-rank} of $A$ is just the rank $r(G/A)$.

\medskip

We now have at our disposal all the ingredients necessary to prove the following statement on perspectivity of divisible torsion-free groups.

\begin{prop}
Let $G$ be a divisible torsion-free group. If $A$ and $A'$ are its two direct summands with equal finite ranks or co-ranks, then $A\overset{P}\sim A'$.
\end{prop}

\begin{proof}
Write $G=A\oplus B=A'\oplus B'$. If $G$ possesses finite rank, then $G$ is perspective and the assertion is proven. Therefore, we further can consider the case where $G$ possesses infinite rank.

First, we consider the case of finite equal co-ranks.

By assumption $r(B)=r(B')=n\in\mathbb{N}$. So, $B\cong B'$. Assuming that $B\cap A'=\{0\}$, we finds $D:=A'\oplus B$
is a direct summand of $G$ such that the co-ranks of $A'$ in $D$ and in $G$ are equal, whence $D=G$ and $G=A'\oplus B=A\oplus B$.

Assume now that $B_1:=B\cap A'\neq\{0\}$ and $B'\cap A\neq\{0\}$. We, thereby, have $B=B_1\oplus B_2$ for some
$B_2\leq B$. Likewise, if $A_1:=A\cap A'$, then $A=A_1\oplus A_2$, $A'=A_1\oplus A_3$ for some $A_2\leq A$, $A_3\leq A'$. Note that $r(A_2)$ is finite; in fact, if $a\in A_2$, then $a=a'+b'$ for some $a'\in A'$, $b'\in B'$, where
$b'\neq 0$ because $A_2\cap A'=\{0\}$, which means that $r(A_2)\leq r(B')=n$.
Furthermore, we write $G=(A_1\oplus B_1)\oplus(A_2\oplus B_2)$ and $A'=(A_1\oplus B_1)\oplus A_1'$, where
$A_1'=A'\cap(A_2\oplus B_2)$ with $A'\cap A_2=A'\cap B_2=\{0\}$. By what we have shown so far, $A_1\oplus B_1$
has finite co-rank. Moreover, if $r(B_2)=k$, then $r(B_1)=n-k$.

Letting $\pi:G\to A_2$ be a projection, if $A_2'=\pi(A_1')$, then $A_2=A_2'\oplus A_2''$ for some $A_2''\leq A_2$.
Thus, $A_1\oplus B_1\oplus A_2'\cong A'$ as $A_1'\cong A_2'$. But, because the co-rank of $A_1\oplus B_1$ is finite and $r(A_2')$ is also finite, then we detect that $r(A_2''\oplus B_2)=n$. Hence, $r(A_2'')=n-k$ and $A_1'\leq A_2'\oplus B_2$, $$G=(A_1\oplus B_1)\oplus(A_2''\oplus(A_2'\oplus B_2)).$$

Since $r(A_2'')=n-k$, there exists an injection $\varphi: B_1\to A_2''$. In conjunction with Lemma~\ref{tp02}, one writes that $B\oplus A_2''=B_1'\oplus A_2''$, $$G=(A_1\oplus B'_1)\oplus(A_2''\oplus(A_2'\oplus B_2))=
A\oplus(B_1'\oplus B_2),$$ where $A=A_1\oplus A''_2\oplus A_2'$.

Notice that $A'\cap(B_1\oplus B_2)=\{0\}$. Indeed, to substantiate this equality, one observes that every element
$x\in B_1\oplus B_2$ has the form $x=b_1+\varphi(b_1)+b_2$ for some $b_1\in B_1$, $b_2\in B_2$. If $x\in A'$, then
$$\varphi(b_1)+b_2\in A'\cap(A_2''\oplus B_2)=A_1'\cap(A_2''\oplus B_2).$$
Besides, every element $y\in A_1'$ has the form $y=a_2'+b_2$ for some $a_2'\in A_2'$,$b_2\in B_2$. So, if $y\in A_2''\oplus B_2$, then $y=a_2''+b_2'$ for some $a_2''\in A_2''$,$b_2'\in B_2$ whence $a_1'+b_2=a_2''+b_2'$ and
$$a_2'-a_2''=b_2'-b_2\in(A_2''\oplus A_1')\cap B=\{0\}.$$ Consequently, $b_2'=b_2$, $a_2'=a_2''\in A_2''\cap A_2'=\{0\}$, i.e., $a_2'=a_2''=0$. Then, $y=b_2\in A_1'\cap B_2=\{0\}$ and, finally, $y=0$. So, $A'\cap(B_1'\oplus B_2)=\{0\}$ as wanted and, in accordance to what we already proven above, it must be that $G=A'\oplus(B_1'\oplus B_2)$.

We are now considering the case of finiteness of the equal ranks of $A$ and $A'$. However, since the rank of $G$ is infinite, we may write that $B_1:=B\cap B'\neq\{0\}$, $B=B_1\oplus B_2$, $B'=B_1\oplus B'_2$ for some $B_2\leq B$,
$B_2'\leq B'$. We also have
$$G=(A\oplus B_2)\oplus B_1=(A'\oplus B'_2)\oplus B_1.$$ If $0\neq x\in B_2'$, then $x=a+b_1+b_2$ for some $a\in A$,
$b_1\in B_1$ and $b_2\in B_2$. Here $a\neq 0$, because $B_2'\cap B=\{0\}$. Thus, if $\alpha:G\to A$ is a projection, then $\alpha\upharpoonright B_2'\to A$ is an injection. Therefore, $r(B_2')\leq n=r(A)$. Similarly, $r(B_2)\leq n$.
Consequently, the subgroup $D:=(A'\oplus B_2')+(A\oplus B_2)$ must have finite rank.

Further, $D=(A'\oplus B_2')\oplus B_1'$, where $B_1'=D\cap B_1$. Analogously, $D=(A\oplus B_2)\oplus B_1'$.
Since we know that the group $D$ is perspective, then $D=A'\oplus U=A\oplus U$ for some $U\leq D$. And since
$G=D\oplus W$ for some $W\leq G$, then $$G=A'\oplus(U\oplus W)=A\oplus(U\oplus W),$$ as desired.
\end{proof}

The last proposition manifestly demonstrates that all divisible torsion-free groups $G$ have the weak resemblance to the standard TP property in the sense that if $G=A\oplus B=A'\oplus B'$, where the co-ranks (resp., the ranks) of $A$ and $A'$ are finite and $A\overset{P}\sim A'$, then $B\overset{P}\sim B'$.

\medskip

In ending our work, we initiate some further discussion into two subsequent sections, hoping it to stimulate an intensive research in the near future on the discussed themes.

\section{Some relationships between IC and TP properties}

We, foremost, recollect the following well-known terminology: a module $M$ is called \textit{purely infinite} if $M\cong M\oplus M$. A module is called \textit{Dedekind finite} (or, in other terms, \textit{directly finite}) if it is {\it not} isomorphic to any proper direct summand (see \cite{mm}). Note that, referring to \cite[p.13]{handelman}, a ring having transitive perspectivity need {\it not} to be Dedekind finite. 

\medskip

Nevertheless, the following result illustrates that in the class of divisible groups, TP implies IC.

\begin{lem}\label{lemmaqc}
Let $G$ is a divisible group. If $G$ is TP, then it is perspective and so IC.
\end{lem}

\begin{proof}
Appealing to \cite[Theorem 3.5]{tp}, $G$ is perspective. Moreover, owing to \cite[Propositions 3.7, 3.13]{perspective-ab}, one notices that the torsion-free rank as well as each $p$-rank of $G$ are finite, and thus we are done.

%\cite[Theorem 2.29]{mm}, $G$ can be written as $G=D\oplus P$, where $D$ is directly finite and $P$ is purely infinite. But, since the weakly perspective condition is inherited via direct summands, $P$ is weakly perspective too. However, \cite[Proposition 9]{wp} is applicable to get that $P=\{0\}$, and hence $G$ is directly finite. Thus, \cite[Proposition 1.28]{mm} employs to deduce that $G$ is IC, as asserted.
\end{proof}

So, we would like to find now an example of a group that is TP but {\it not} IC; notice that a CP group is {\it always} IC (see point (4) of the discussion alluded to above) and the reciprocal is {\it not} true in virtue of Note\ref{notenew} quoted above. 

\medskip

However, the class of groups with the substitution property is {\bf not} closed under formation of {\bf infinite} direct sums. In fact, if $A$ is a module such that $A \not\cong  A\oplus A$ and is an infinite ordinal, then $A^{(\alpha)}$ does {\it not} have the substitution property: indeed, one obviously sees that $$N=A^{(\alpha)}\oplus\{0\}=A^{(\alpha-1)}\oplus A$$ and $$A^{(\alpha)}\cong A^{(\alpha-1)}$$ but
$$C\oplus\{0\}\neq C\oplus A$$ for all $C\leq N$. Same holds for $A^{\alpha}$ if $\alpha$ is infinite.

\section{Perspectively finite modules and groups}

A module $M$ is called {\it perspectively finite} (or just {\it PF} for short) if $M$ is not perspectively related to a proper non-zero summand of itself; i.e., for any proper non-zero direct summand $D$ of $M$, there is no
submodule $E$ of $M$ such that $M \oplus E\cong D \oplus E$. Two further cases are possible:

\medskip

(i) If $E$ is a direct summand of $M$.

\medskip

(ii) $D\oplus E$ is a (proper) direct summand of $M$ (in fact, which differs $M$). In this case, DF modules are always PF, because $M$ will have an isomorphic proper direct summand.

\medskip

We end our work with the following challenging question which, at this stage, seems to be extremely difficult.

\medskip

\noindent{\bf Problem.} If $G$ is a homogeneous separable torsion-free group of infinite rank and type $t$, and $R$ is a rank one torsion-free group of the same type $t$, does it follow that the groups $G$ and $G\oplus R$ are themselves isomorphic?

\bigskip

\noindent{\bf Funding:} The work of the first-named author, A.R. Chekhlov, was supported by the Ministry of Science and Higher Education of Russia (agreement No. 075-02-2025-1728/2) and by the Regional Scientific and Educational Mathematical Center of Tomsk State University.

\end{document}